\documentclass [12pt,reqno]{amsart}

\usepackage{amsmath,amsfonts,amssymb,amscd,verbatim,delarray,verbatim,calrsfs}

\def\Q{{\mathbb Q}}

\def\Z{{\mathbb Z}}
\def\C{{\mathbb C}}
\def\R{{\mathbb R}}
\def\F{{\mathbb F}}

\def\Gal{\mathrm{Gal}}

\def\ord{\mathrm{ord}}

\def\UU{\mathrm{U}}
\def\Fr{\mathrm{Fr}}

\def\End{\mathrm{End}}

\def\cl{\mathrm{cl}}

\def\OO{{\mathcal O}}

\def\UU{\mathrm{U}}

\def\is{\mathrm{isg}}

\def\HH{\mathbf{H}}

        \def\K_a{\bar{K}}

\def\dim{\mathrm{dim}}

\newtheorem{thm}{Theorem}[section]
\newtheorem{lem}[thm]{Lemma}
\newtheorem{cor}[thm]{Corollary}
\newtheorem{prop}[thm]{Proposition}
\theoremstyle{definition}
\newtheorem{defn}[thm]{Definition}
\newtheorem{ex}[thm]{Example}

\newtheorem{rem}[thm]{Remark}

\newtheorem{sect}[thm]{}

\title[Isogeny classes of abelian varieties ]{Isogeny classes and Endomorphisms algebras of abelian varieties over finite fields}
\author{Yuri G. Zarhin}
\address{Department of Mathematics, Pennsylvania
State University, University Park, PA 16802, USA}
\email{zarhin\char`\@math.psu.edu}

\large

\begin{document}
\begin{abstract}
We construct non-isogenous simple ordinary abelian varieties over an algebraic closure of a finite field  with isomorphic endomorphism algebras.

\end{abstract}

\thanks{The author  was partially supported by Simons Foundation Collaboration grant   \# 585711.
Part of this work was done during his stay in January - May 2022 at the Max-Planck Institut f\"ur Mathematik (Bonn, Germany),
whose hospitality and support are gratefully acknowledged.}

\maketitle

\section{Introduction}
\begin{sect}
\label{zero}
If $K$ is a number field then we write $\mathrm{Cl}(K)$ for the (finite commutative) ideal class group of $K$, 
$\cl(K)$ for the class number  of $K$ (i.e., the cardinality of  $\mathrm{Cl}(K)$) and $\exp(K)$ for the exponent of $\mathrm{Cl}(K)$. 
 Clearly,  $\exp(K)$ divides  $\cl(K)$. (The equality holds if and only if $\mathrm{Cl}(K)$  is cyclic, which is not always the case, see \cite[Tables]{BS}.)
 In addition, 
$\exp(K)$ is odd if and only if   $\cl(K)$ is odd. We write $\OO_K$ for the ring of integers in $K$ and $\UU_K$ for the group of {\sl units}, 
i.e., the multiplicative group of invertible elements in $\OO_K$.  As usual, an element of $\UU_K$  is called a unit in $K$ or a $K$-unit. It is well known (and can be easily checked) that if a unit $u$ in $K$ is a square in $K$ then it is also a square in $\UU_K$.

Let $p$ be a prime and  $q$ a positive integer that is a power of $p$.
We write $\F_p$ for the $p$-element finite field and $\F_q$ for its $q$-element overfield. As usual, $\bar{\F}_p$ stands for an algebraic closure of $\F_p$,
which is also an algebraic closure of $\F_q$.  We have
$$\F_p\subset \F_q\subset \bar{\F}_p.$$
If $X$ is an abelian variety over $\bar{\F}_p$ then we write $\End^{0}(X)$ for its endomorphism algebra 
$\End(X)\otimes\Q$, which is a finite-dimensional semisimple algebra over the field $\Q$ of rational numbers.
If $X$ is defined over $k=\F_q$ then we write $\End_k(X)$ for its ring of $k$-endomorphisms and
$\End_k^{0}(X)$ for the $\Q$-algebra 
$\End_k(X)\otimes\Q$; one may view $\End_k^{0}(X)$ as the $\Q$-subalgebra of $\End^{0}(X)$ with the same $1$.

It is well known that isogenous abelian varieties have isomorphic endomorphism algebras and the same dimension (and $p$-adic Newton polygon).
In addition, an abelian variety is simple if and only if its endomorphism algebra is a division algebra over $\Q$.
It is also known (Grothendieck-Tate) that $\End^{0}(X)$  uniquely determines the dimension of $X$ \cite{O}.
Namely, $2\dim(X)$ is the maximal $\Q$-dimension of a semisimple commutative $\Q$-subalgebra of $\End^0(X)$. However,
it turns out that there are non-isogenous abelian varieties over $\bar{\F}_p$ with isomorphic endomorphism algebras.
\end{sect}
The aim of this note is to provide explicit examples of such varieties.

Let me start with a classical result of   M. Deuring  about elliptic curves  \cite{Deuring},
\cite[Ch. 4]{Waterhouse}.  
 
\begin{prop}
\label{elliptic}
Let $K$ be an imaginary quadratic field.

\begin{itemize}
\item[(i)]
Let $p$ be a prime and $E$ an elliptic curve over $\bar{\F}_p$ such that
$\End^0(E)$ is isomorphic to $K$.

Then $p$ splits  in $K$ and $E$ is ordinary.
\item[(ii)]
Let $p$ be a prime that splits in $K$. 

Then 
all the elliptic curves $E$ over 
$\bar{\F}_p$ with $\End^0(E)\cong K$ are mutually isogenous.
\end{itemize}
\end{prop}

I did not find  in the literature the following assertion  that complements Proposition \ref{elliptic}.

\begin{prop}
\label{ellipticBis}
Let $K$ be an imaginary quadratic field and $p$
 a prime that splits in $K$. Let us put $q=p^{\exp(K)}$.
 
Then there exists an elliptic curve $E$ that is defined with all its endomorphisms
over $\F_{q}$ and such that $\End^0(E) \cong K$.
\end{prop}

\begin{rem}
One may deduce from (\cite[Satz 3]{Deuring2}, \cite[Sect. 6, Cor. 1 on p. 507]{ST}) 
that if we put $q_1=p^{\cl(K)}$ then  there exists an elliptic curve $E$ that is defined with all its endomorphisms
over $\F_{q_1}$ and such that $\End(E) \cong \OO_K$ (and therefore $\End^0(E) \cong K$).

\end{rem}

The next result is an analogue of Proposition \ref{elliptic} for abelian surfaces and quartic fields.

\begin{prop}
\label{surface}
Let $K$ be a CM quartic field that is a cyclic extension of $\Q$.

\begin{itemize}
\item[(i)]
Let $p$ be a prime and $Y$ an abelian surface over $\bar{\F}_p$ such that
$\End^0(Y)$ is isomorphic to $K$.

Then $p$ splits completely in $K$ and $Y$ is simple ordinary.
\item[(ii)]
Let $p$ be a prime that splits in $K$. 

Then all the abelian surfaces $Y$ over 
$\bar{\F}_p$ with $\End^0(Y)\cong K$ are mutually isogenous.
In addition, there exists such an $Y$ that is defined with all its endomorphisms
over $\F_{p^{2c}}$ where $c=\exp(K)$.
\end{itemize}
\end{prop}

Our main result is the following assertion.

\begin{thm}
\label{main}
Let $n$ be a positive integer and $K$ is a CM field that is a cyclic degree $2^n$ extension of $\Q$. Let $K_0$ be the only degree $2^{n-1}$ subfield of $K$, which is the maximal totally real subfield of $K$. Let us put $c=\exp(K)$.
\begin{itemize}
\item[(i)]
Let $p$ be a prime and $A$ an abelian variety over $\bar{\F}_p$ such that
$\End^0(A)$ is isomorphic to $K$.

Then $p$ splits  completely in $K$ and $A$ is an ordinary simple abelian variety of dimension $2^{n-1}$.
\item[(ii)]
Let $p$ be a prime that splits completely in $K$. Let us put $q=p^c$.

\begin{enumerate}
\item[(1)]
There are precisely $2^{2^{n-1}-n}$ isogeny classes of abelian varieties $A$ over $\bar{\F}_p$,
whose endomorphism algebra 
$\End^{0}(A)$
is isomorphic to $K$. 
\item[(2)]
 Each of these isogeny classes contains an abelian variety that is defined with all its endomorphisms over $\F_{q^2}$.
 \item[(3)]
 Assume additionally that every totally positive unit in $K_0$ is a square in $K_0$. 
 
 Then 
 each of these isogeny classes contains an abelian variety that is defined with all its endomorphisms over $\F_{q}$.
 \end{enumerate}
\end{itemize}
\end{thm}

\begin{rem}
\begin{itemize}
\item[(a)]
If $n=1$ then $K$ is an imaginary quadratic field and therefore $K_0=\Q$ and $\UU_{\Q}=\{\pm 1\}$. The only (totally) positive unit in $\Q$ is $1$, which is obviously a square in $\Q$.
Hence, Propositions \ref{elliptic} and \ref{ellipticBis} are the special case of Theorem \ref{main} with $n=1$. On the other hand,  Proposition \ref{surface} follows  readily from the special case of Theorem \ref{main} with $n=2$.  
\item[(b)]
If $n \ge 3$ then the number  $2^{2^{n-1}-n}$ of the corresponding isogeny classes is strictly greater than $1$. This gives us examples of non-isogenous abelian varieties over $\bar{\F}_p$, whose endomorphism algebras are isomorphic to $K$ and therefore are mutually isomorphic.
\item[(c)]
Now let $n$ be an arbitrary positive integer.
By Chebotarev's density theorem, the set of primes that split completely in $K$ is infinite (and even has a positive density $1/2^n$).
 \end{itemize}
 \end{rem}

\begin{cor}
\label{mainFermat}
Let $r$ be a Fermat prime (e.g., $r=3,5,17,257, 65537$). Let $p$ be a prime that is congruent to $1$ modulo $r$. Let us put
\begin{equation}
\label{isg}
\is(r)=\frac{2^{(r-1)/2}}{(r-1)}.
\end{equation}
Then there are precisely $\is(r)$ isogeny classes of simple $(r-1)/2$-dimensional ordinary abelian varieties $A$ over $\bar{\F}_p$,
whose endomorphism algebra 
$$\End^{0}(A)=\End(A)\otimes\Q$$
is isomorphic to the $r$th cyclotomic field $\Q(\zeta_r)$.
 In addition, if $c=\exp(\Q(\zeta_r))$  and $q=p^c$
 then each of these isogeny classes contains an abelian variety that is defined with all its endomorphisms over $\F_{q}$.
\end{cor}

\begin{rem}
\label{splitP}
The congruence condition on $p$ means that $p$ splits completely in $\Q(\zeta_r)$.  There are infinitely many such $p$, thanks to Dirichlet's theorem
about primes in an arithmetic progression. More precisely, the set of such primes has density $1/(r-1)$.
\end{rem}

\begin{rem}
\label{invIsog}
It is well known that the property of being simple (resp. ordinary) is invariant under isogenies. 
\end{rem}

\begin{rem}
\label{oneISG}
Let $r$ be a Fermat prime. Clearly,  $\is(r)=1$ if and only if $r \le 5$.

 Let $p$ be a prime $p$ that is congruent to $1\bmod r$. It follows from Theorem \ref{main} that $r \le 5$ if and only if
there is a precisely one isogeny class of simple ordinary $(r-1)/2$-dimensional abelelian varieties over $\bar{\F}_p$, whose endomorphiam algebra 
is isomorphic to $\Q(\zeta_r)$. In other words, all such abelian varieties are mutually isogenous over $\bar{\F}_p$, if and only if $r \in \{3,5\}$.
\end{rem}

\begin{ex}
\begin{itemize}
\item[(i)] Take $r=3$. 
We have
$\is(3)=1$. It follows from Remark \ref{oneISG} that
if $p \equiv 1 \bmod 3$ then
all ordinary elliptic curves over  $\bar{\F}_p$ with endomorphism algebra $\Q(\zeta_3)$ are isogenous.
(This assertion seems to be well known.)  This implies that each such elliptic curve is isogenous over $\bar{\F}_p$
to $y^2=x^3-1$.
\item[(ii)]
Take $r=5$. 
We have
$\is(5)=1$. 
It follows from Remark \ref{oneISG} that if $p \equiv 1 \bmod 5$ then
all abelian varieties over  $\bar{\F}_p$ with endomorphism algebra $\Q(\zeta_5)$ are two-dimensional simple ordinary and mutually isogenous.
This implies that each such abelian variety is isogenous to the jacobian of the genus $2$  curve $y^2=x^5-1$.
\end{itemize}
\end{ex}

\begin{ex}
Let us take $r=17$.  Then $\cl(\Q(\zeta_{17}))=1$ \cite{W}. Let us choose a prime $p$ that is congruent to $1$ modulo $17$ (e.g., $p=103$).
We have
$$\is(17)=\frac{2^{8}}{16}=16.$$
By Theorem \ref{main}, there are precisely $16$ isogeny classes of simple ordinary $\frac{16}{2}=8$-dimensional abelian varieties over
$\bar{\F}_p$ with endomorphism algebras $\Q(\zeta_{17})$. In addition, each of these isogeny classes contains an abelian eightfold
that is defined with all its endomorphisms over $\F_{p}$. 

This implies  that there exist sixteen
$8$-dimensional ordinary simple abelian varieties $A_1, \dots, A_{16}$ over $\bar{\F}_p$ that are  mutually {\sl non}-isogenous but each endomorphism algebra
$\End^0(A_i)$ is isomorphic to  $\Q(\zeta_{17})$  (for all $i$ with $1 \le i \le 16$). In particular, 
$$\End^0(A_i) \cong \End^0(A_j) \ \forall i,j \ (1 \le i<j \le 16).$$
In addition, each $A_i$ and all its endomorphisms are defined over $\F_{p}$.
This gives an answer to a question of L. Watson \cite{Watson}.
\end{ex}

The following assertion is a natural generalization of Corollary \ref{mainFermat}.
\begin{cor}
\label{cycleP}
Let $r$ be an odd prime and
$(r-1)=2^n \cdot m$ where $n$ is a positive integer and $m$ is a positive odd integer. Let $\HH$ be the only order $m$ subgroup
of the cyclic Galois group
$$\Gal(\Q(\zeta_r)/\Q)=(\Z/r\Z)^{*}$$
of order $(r-1)$.
Let 
\begin{equation}
\label{Kr}
K=K^{(r)}:=\Q(\zeta_r)^{\HH}
\end{equation}
 be the subfield of $\HH$-invariants in $\Q(\zeta_r)$.

Then:
\begin{itemize}
\item[(0)]
$K^{(r)}$ is a CM field  that is a cyclic degree $2^n$ extension of $\Q$.  In addition,  a prime $p$ splits completely in $K^{(r)}$ if and only if $p \ne r$ and
$p \bmod r$ is a $2^n$th power in $\F_r$.
\item[(i)]
Let $p$ be a prime and $A$ an abelian variety over $\bar{\F}_p$ such that
$\End^0(A)$ is isomorphic to $K^{(r)}$.

Then $p$ splits  completely in $K^{(r)}$ and $A$ is an ordinary simple abelian variety of dimension $2^{n-1}$.
\item[(ii)]
Let $p$ be a prime that splits completely in $K^{(r)}$ and let $q=p^c$ where $c=\exp(K^{(r)})$.

 Then there are precisely $2^{2^{n-1}-n}$ isogeny classes of abelian varieties $A$ over $\bar{\F}_p$,
whose endomorphism algebra 
$\End^{0}(A)$
is isomorphic to $K^{(r)}$.
 In addition,  each of these isogeny classes contains an abelian variety that is defined with all its endomorphisms over $\F_{q}$.
 
\end{itemize}
\end{cor}

\begin{rem}
\label{Kr0}
Let $K=K^{(r)}$. It is well known that $r$ is totally ramified in $\Q(\zeta_r)$  and therefore in its subfield $K$ as well. This implies that if $K_0$ is the only degree $2^{n-1}$ subfield of $K$, which is the maximal totally real subfield of $K$,  then the quadratic extension $K/K_0$ is {\sl ramified}.  On the other hand, it is known that
(\cite[Sect. 38]{Hasse}, \cite[p. 77-78]{CH}) that $\cl(K^{(r)})$ is {\sl odd}  (and therefore $c=\exp(K^{(r)})$ is also odd).  It follows from  \cite[Sect. 37, Satz 42]{Hasse} (see also \cite[Cor. 13.10 on p. 76 ]{CH}) that
 $K_0$ has {\sl units with independent signs} (there are units of $K_0$ of every possible signature), which implies (thanks to \cite[Lemma 12.2 on p. 55]{CH}) that every {\sl totally positive} unit in $K_0$ is a square in $K_0$ and therefore is a square in $\UU_{K_0}$.
\end{rem}

\begin{ex}
Let us fix an integer $n \ge 2$. Here is a construction of infinitely many mutually non-isomorphic CM fields that are cyclic degree $2^n$ extensions of $\Q$.  Let us consider the infinite (thanks to Dirichlet's theorem) set of primes $r$ that are congruent to $1+2^n$ modulo $2^{n+1}$.  Then
$r-1=2^n \cdot m$ where $m$ is an odd positive integer. In light of Corollary \ref{cycleP}, the corresponding
subfield $K^{(r)}$ of $\Q(\zeta_r)$  defined by \eqref{Kr} enjoys the desired properties.
Since $K^{(r)}$ is a subfield of  $\Q(\zeta_r)$, the field extension  $K^{(r)}/\Q$ is ramified precisely at $r$. This implies that the fields $K^{(r)}$ are mutually non-isomorphic (and even linearly disjoint) for distinct $r$.
\end{ex}

The paper is organized as follows.  In Section \ref{prel} we review basic results about maximal ideals of $\OO_K$.
In Section \ref{ordinaryS} we concentrate on so called {\sl ordinary} Weil's $q$-numbers in $K$.
In Section \ref{WeilKK} we discuss simple abelian varieties over $\F_q$, whose Weil's numbers lie in $K$. 
In Section \ref{HondaTate} we discussed some basic facts of Honda-Tate theory \cite{T66,Honda,T}.
Section \ref{mainProof} contains proofs of main results.

In what follows we will freely use the following elementary well known observation. {\sl Any $\Q$-subalgebra with $1$ of a number field $K$ is actually a subfield of $K$; in particular, it is also a number field. E.g., if $u$ is an element of $L$ then the subfield $\Q(u)$ generated by $u$ coincides with the $\Q$-subalgebra $\Q[u]$ generated by $u$.}

{\bf Acknowledgements}. I am grateful to Ley Watson  for an interesting stimulating question \cite{Watson}.

\section{Preliminaries}

\begin{sect}
\label{not} We keep the notation  and assumptions of Subsection \ref{zero} and Theorem \ref{main}.  As usual,  $\Q,\R,\C$ stand for the fields of rational, real and complex numbers and $\bar{\Q}$ for the (algebraically closed) subfield of all algebraic numbers in $\C$. We write $\Z$ (resp. $\Z_{+}$) for the ring of integers (resp. for the additive semigroup of {\bf nonnegative} integers). If $T$ is a finite set then we write $\#(T)$ for the number of its elements.

Recall \cite{Honda,T} that an algebraic integer $\pi\in \bar{\Q}$ is called a {\sl Weil's $q$-number} if all its Galois-conjugates have the archimedean absolute value $\sqrt{q}$.

Throughout this paper, $n$ is a positive integer and $K$ is a CM field that is a degree $2^n$ cyclic extension of $\Q$.
 We view $K$ as a subfield of  $\C$; in particuar,  $K$ is a subfield of $\bar{\Q}$ that is stable under the {\sl complex conjugation}. We denote by
$$\rho: K \to K$$
the restriction of the complex conjugation to $K$; one may view $\rho$ as an element of  order 2 in the Galois group
$$G:=\Gal(K/\Q)$$
where $G$ is a cyclic group of order $2^n$.

\begin{rem}
\label{useful}
Let $\pi \in K\subset \C$. 
\begin{itemize}
 \item 
 Suppose that $\pi$ is a Weil's $q$-number.  Then $\pi$ is a algebraic integer, i.e.,
 $\pi   \in \OO_K$. Since the absolute value of $\pi$ is the square root of $q$, we have
 $\pi \cdot \rho(\pi)=q$.
\item
Conversely, suppose that  
$\pi   \in \OO_K$ (i.e., $\pi$ is an algebraic integer) and 
 \begin{equation}
\label{Riemann}
\pi \cdot \rho(\pi)=q
\end{equation}
Since $K/\Q$ is Galois,   all the Galois-conjugates of $\pi$  also lie in $\OO_K$
and constitute the orbit 
$$G \pi=\{\sigma(\pi)\mid \sigma \in G\}$$ of 
$G$.  Since $G$ is commutative and contains $\rho$, it follows from \eqref{Riemann}
that  for all $\sigma \in G$ 
$$\sigma(\pi)\cdot \rho(\sigma(\pi))=\sigma(\pi)\cdot \sigma(\rho(\pi))=\sigma(\pi\cdot \rho(\pi))=\sigma(q)=q.$$ 
\end{itemize}

It follows readily  that {\it $\pi \in K$ is a Weil's $q$-number if and only if $\pi \in \OO_K$ and \eqref{Riemann} holds.}
\end{rem}

We write $W(q,K)$ for the set of Weil's $q$-numbers in $K$ and $\mu_K$ for the (finite cyclic) multiplicative
group of roots of unity in $K$. Clearly, $W(q,K)$ is a finite $G$-stable subset of $\OO_K$, which is also stable under
multiplication by elements of $\mu_K$. The latter gives rise to the free action of $\mu_K$ on $W(q,K)$  defined by
$$\mu_K \times W(q,K) \to W(q,K), \ \zeta, \pi \mapsto \zeta \pi \ \forall \zeta\in \mu_K, \pi \in W(q,K).$$
\begin{rem}
\label{kron}
It is well known (and follows easily from a theorem of Kronecker \cite[Ch. IV, Sect. 4, Th.8]{Weil}) that $\pi_1,\pi_2 \in  W(q,K)$ lie in the same $\mu_K$-orbit
(i.e., $\pi_2/\pi_1$ is a root of unity) if and only if the  ideals $\pi_1\OO_K$ and $\pi_2\OO_K$ of $\OO_K$ do coincide.
\end{rem}
\end{sect}

\label{prel}
Recall (Subsection \ref{not}) that $K$ is a subfield of the field $\C$ of complex numbers that is stable under the complex conjugation.
Then 
$$K_0:=K \bigcap\R$$
is a (maximal) {\sl totally real} number (sub)field, whose degree $[K_0:\Q]$ is
$$\frac{[K:\Q]}{2}=\frac{2^n}{2}=2^{n-1}.$$

\begin{sect}
\label{Gideals}
Recall that
the Galois group
$G=\Gal(K/\Q)$
is a cyclic group of order $2^n$. Hence, it has precisely one element of order $2$ and therefore this element must coincide with the {\sl complex conjugation} 
$$\rho: K \to K.$$
The properties of $G$ imply that every nontrivial subgroup $H$ of $G$ contains $\rho$.
It follows that every proper subfield of $K$ is {\sl totally real}. Indeed, each such subfield is the
subfield $K ^H$ of $H$-invariants for a certain nontrivial subgroup $H$ of $G$. Since $H$ contains $\rho$, the subfield
$K^H$ consists of $\rho$-invaraiants and therefore is totally real; in particular,
$$K^H \subset \R.$$
\end{sect}

\begin{sect}
\label{freeGaction}
Let $\ell$ be a prime and $S(\ell)$ be the set of maximal ideals $\mathfrak{P}$ of $\OO_K$
that divide $\ell$. Since $K/\Q$ is a Galois extension,  $G$ acts transitively on $S(\ell)$. In particular, $\#(S(\ell))$ divides $\#(G)=2^n$.
This implies that if $\ell$ {\sl splits completely} in $K$, i.e., 
$$\#(S(\ell))=2^n=\#(G)$$
 then the action of $G$ on $S(\ell)$ is {\sl free}. 

On the other hand, if  a prime $\ell$ does {\sl not} split completely in $K$, i.e., 
$$\#(S(\ell))<2^n=\#(G),$$
 then the action of $G$ on $S(\ell)$ is {\sl not} free. 
Let $H(\ell)$ be the stabilizer of any $\mathfrak{P}\in S(\ell)$, which does not depend on a choice of $\mathfrak{P}$, because $G$ is commutative.
Then $H(\ell)$ is a nontrivial subgroup of $G$ and therefore contains $\rho$, i.e.,
$$\rho(\mathfrak{P})=\mathfrak{P} \ \forall \mathfrak{P} \in S(\ell)$$
if $\ell$ does {\sl not} split completely in $K$.

Let $e(\ell)$ be the {\sl ramification index} in $K/\Q$ of  $\mathfrak{P}\in S(\ell)$, which does {\sl not} depend on $\mathfrak{P}$,
because $K/\Q$ is Galois. We have the equality of ideals
\begin{equation}
\label{ramification}
\ell \OO_K=\prod_{\mathfrak{P}\in S(\ell)}\mathfrak{P}^{e(\ell)}.
\end{equation}
It follows that $K/\Q$ is {\sl unramified} at $\ell$ if and only if $e(\ell)=1$.
We write
\begin{equation}
\label{ordP}
\mathrm{ord}_{\mathfrak{P}}: K^{*} \twoheadrightarrow \Z
\end{equation}
for the discrete valuation map attached to $\mathfrak{P}$. We have
\begin{equation}
\label{ramE}
\mathrm{ord}_{\mathfrak{P}}(\ell)=e(\ell) \ \forall \mathfrak{P}\in S(\ell);
\end{equation}
\begin{equation}
\label{ordOO}
\mathrm{ord}_{\mathfrak{P}}(u)\ge 0 \ \forall  u \in \OO_R\setminus\{0\}, \mathfrak{P}\in S(\ell);
\end{equation}
\begin{equation}
\label{ordrho}
\mathrm{ord}_{\mathfrak{P}}(\rho(u))=\mathrm{ord}_{\rho(\mathfrak{P})}(u) \ \forall  u \in K^{*}, \mathfrak{P}\in S(\ell).
\end{equation}
\end{sect}

\begin{sect}
Let $p$ be a prime, $j$ a positive integer, and $q=p^j$. 

Let $\pi \in O_K$ be a Weil's $q=p^j$-number.
Let us consider the ideal $\pi\OO_K$ in $\OO_K$. Then there is a nonnegative integer-valued function
\begin{equation}
\label{mPi}
d_{\pi}: S(p) \to \Z_{+},  \ \mathfrak{P} \mapsto d_{\pi}(\mathfrak{P}):=\mathrm{ord}_{\mathfrak{P}}(\pi)
\end{equation}
such that
\begin{equation}
\label{prodPi}
\pi\OO_K=\prod_{\mathfrak{P}\in S(p)}\mathfrak{P}^{d_{\pi}(\mathfrak{P})}.
\end{equation}
It follows from \eqref{Riemann} that
\begin{equation}
\label{mPIrho}
d_{\pi}(\mathfrak{P})+d_{\pi}(\rho(\mathfrak{P}))=\mathrm{ord}_{\mathfrak{P}}(q)=j \cdot e(\ell) \ \forall \mathfrak{P}\in S(p).
\end{equation}
\end{sect}

\begin{lem}
\label{nonsplit}
Let $\pi \in O_K$ be a Weil's $q=p^j$-number. If $p$ does not split completely in $K$
then $\pi^2/q$ is a root of unity.
\end{lem}

\begin{proof}

Since  $p$ does not split completely in $K$,  it follows from arguments of Subsection \ref{Gideals} that
$$\rho(\mathfrak{P})=\mathfrak{P} \ \forall \mathfrak{P}\in S(p).$$
It follows from \eqref{mPIrho} that
$$d_{\pi}(\mathfrak{P})=\frac{j\cdot e(p)}{2}  \ \forall \mathfrak{P}\in S(p);$$
in particular, $j$ is {\sl even} if $e(p)=1$ (i.e., if $K/\Q$ is {\sl unramified}  at $p$).
This implies that $\pi^2/q$ is a $\mathfrak{P}$-adic unit for all $\mathfrak{P}\in S(p)$.
On the other hand, it follows from \eqref{Riemann} that $\pi^2/q$ is an $\ell$-adic unit
for all primes $\ell \ne p$. It follows from the very definition of Weil's numbers that 
$$|\sigma\left(\pi^2/q\right)|_{\infty}=1 \ \forall \sigma \in G.$$
(Here
$| \ |_{\infty}:\C \to \R_{+}$ is the standard archimedean value on $\C$.)
 Now it follows from a classical theorem of Kronecker
\cite[Ch. IV, Sect. 4, Th. 8]{Weil} that $\pi^2/q$ is a root of unity.
\end{proof}

\begin{lem}
\label{PsplitK}
Suppose that a prime $p$ completely splits in $K$. (In particular, $K/\Q$ is unramified at $p$.)
 Let $\pi \in O_K$ be a Weil's $q=p^j$-number.
 
 Then either  $\Q(\pi)=K$ or $j$ is even and  $\pi=\pm p^{j/2}$ . 
\end{lem}

\begin{proof}
So, $K/\Q$  is unramified at $p$, i.e., $e(p)=1$ and
\begin{equation}
\label{pOK}
p\OO_K=\prod_{\mathfrak{P}\in S(p)}\mathfrak{P}.
\end{equation}
This implies that
\begin{equation}
\label{qj}
q\OO_K=\prod_{\mathfrak{P}\in S(p)}\mathfrak{P}^{j}.
\end{equation}
Since $p$ splits completely in $K$, the group $G$ acts freely on $S(p)$, in light of Subsection \ref{freeGaction}. In particular,
\begin{equation}
\label{PrhoP}
\mathfrak{P} \ne \rho(\mathfrak{P}) \ \forall \mathfrak{P}\in S(p).
\end{equation}

If the subfield $\Q(\pi)$ of $K$ does {\sl not} coincide with $K$ then it is {\sl totally real}, thanks to arguments of Subsection \ref{Gideals}.
This implies that $\rho(\pi)=\pi$. It follows from \eqref{Riemann} that $\pi^2=q$, i.e., $\pi=\pm p^{j/2}$. This implies that
the ideal $q\OO_K$ is a {\sl square}. It follows from \eqref{qj} that $j$ is {\sl even}.
\end{proof}

\begin{sect}
Suppose that a prime $p$ completely splits in $K$. 
\begin{defn}
 Let $\pi \in O_K$ be a Weil's $q=p^j$-number. We say that $\pi$ is an {\sl ordinary} Weil's $q$-number if the ``slope''
 $\mathrm{ord}_{\mathfrak{P}}(\alpha)/\mathrm{ord}_{\mathfrak{P}}(q)$ is an
 {\sl integer} for all $\mathfrak{P}\in S(p)$.
\end{defn}

It (is well known and) follows from \eqref{Riemann}, \eqref{ordOO} and \eqref{ordrho} that
if $\pi$ is an  ordinary Weil's $q$-number then
\begin{equation}
\label{ordinary01}
\frac{\mathrm{ord}_{\mathfrak{P}}(\pi)}{\mathrm{ord}_{\mathfrak{P}}(q)}=0 \ \text{ or } 1.
\end{equation}
\end{sect}

\section{Equivalence classes of ordinary Weil's $q$-numbers}
\label{ordinaryS}

Let $p$ be a prime that splits completely in $K$.  Throughout this section, by Weil's numbers we mean
Weil's $q$-numbers where $q$ is a power of $p$.  We write $W(q,K)$ for the set of Weil's $q$-numbers in $K$.
We write $\mu_K$ for the (finite cyclic) multiplicative group of roots of unity in $K$.


\begin{defn}
Let $q$ and $q^{\prime}$ be integers $>1$ that are integral powers of $p$.
Let $\pi \in K$ (resp.  $\pi^{\prime}\in K$) be a Weil's $q$-number (resp.   Weil's $q^{\prime}$-number).
Following Honda \cite{Honda}, we say that $\pi$ and $\pi^{\prime}$ are equivalent,
if there are positive integers $a$ and $b$ such that
$\pi^a$ is Galois-conjugate to ${\pi^{\prime}}^b$.
\end{defn}
Clearly, if  $\pi$ and $\pi^{\prime}$ are equivalent then $\pi$ is ordinary if and only if $\pi^{\prime}$ is ordinary.
In order to classify ordinary Weil's numbers in $K$ up to equivalence, we introduce the following notion that is inspired
by the notion of CM type for complex abelian varieties \cite{Shimura} (see also \cite[Sect. 1, Th. 2]{Honda} and  \cite[Sect. 5]{T}).

\begin{defn}
We call a  subset $\Phi\subset S(p)$ a {\sl $p$-type} if $S$ is a disjoint union of $\Phi$ and $\rho(\Phi)$.

Clearly, $\Phi\subset S(p)$ is a $p$-type  if and only if the following two conditions hold (recall that $[K:\Q]=2^n$).

\begin{itemize}
\item[(i)]  $\#(\Phi)=2^{n-1}$.
\item[(ii)] If $\mathfrak{P} \in \Phi$ then $\rho(\mathfrak{P}) \not \in \Phi.$

It is also clear that $\Phi\subset S(p)$ is a $p$-type if and only if $\rho(\Phi)$ is a $p$-type.

\end{itemize}
\end{defn}

Let $H(p)$ be the set of nonzero ideals $\mathfrak{B}$ of $\OO_K$ such that 
$$\mathfrak{B}\cdot \rho(\mathfrak{B})=p \cdot \OO_K.$$ 
In light of \eqref{pOK} and \eqref{PrhoP},
an ideal  $\mathfrak{B}$ of $\OO_K$ lies in $H(p)$ if and only if there exists a 
$2^{n-1}$-element subset $\Phi=\Phi(\mathfrak{B})$ of $H(p)$ that meets every $\rho$-orbit of $S(p)$ at exactly one place and
\begin{equation}
\label{HTB}
\mathfrak{B}=\prod_{\mathfrak{P}\in \Phi(\mathfrak{B})}\mathfrak{P}.
\end{equation}
Such a $\Phi(\mathfrak{B})$ is uniquely determined by $\mathfrak{B}\in H(p)$: namely, it coincides with the set of maximal ideals in $\OO_K$ that contain $\mathfrak{B}$. This implies that
\begin{equation}
\label{cardH}
\#(H(p))=2^{2^{n-1}}.
\end{equation}
Clearly,
\begin{equation}
\label{GaloisH}
\Phi(\sigma(\mathfrak{B}))=\sigma(\Phi(\mathfrak{B})) \ \forall \sigma \in G.
\end{equation}

\begin{lem}
\label{ordinaryPhi}
Let $m$ be a positive integer and $\pi$ be a Weil's $q=p^m$-number in $K$. Then the following conditions are equivalent.
\begin{itemize}
\item[(i)]
$\pi$ is ordinary.
\item[(ii)]
There exists an ideal $\mathfrak{B} \in H(p)$
such that
\begin{equation}
\label{HpPi}
\pi\OO_K=\mathfrak{B}^m.
\end{equation}
\item[(iii)]
The subset 
\begin{equation}
\label{PsiPiOrd}
\Psi(\pi):=\{\mathfrak{P}\in S(p)\mid  \frac{\mathrm{ord}_{\mathfrak{P}}(\pi)}{\mathrm{ord}_{\mathfrak{P}}(q)}=1\}
\end{equation}
is a $p$-type.
\end{itemize}
If these equivalent conditions hold then such an ideal $\mathfrak{B}$ is unique and
$$\Phi(\mathfrak{B})=\Psi(\pi).$$
\end{lem}

\begin{proof}
We have
\begin{equation}
\label{piProd}
\pi\OO_K=\prod_{\mathfrak{P}\in S(p)}\mathfrak{P}^{d(\mathfrak{P})},
\end{equation}
for some $d(\mathfrak{P}) \in \Z_{+}$ such that
\begin{equation}
\label{dSUMrho}
d(\mathfrak{P})+d(\rho(\mathfrak{P}))=m, 
\end{equation}
\begin{equation}
\label{simplification}
\frac{\mathrm{ord}_{\mathfrak{P}}(\pi)}{\mathrm{ord}_{\mathfrak{P}}(q)}=\frac{d(\mathfrak{P})}{m} \   \ \forall  \mathfrak{P}\in S(p).
\end{equation}
This implies that
\begin{equation}
\label{PsiD}
\Psi(\pi):=\{\mathfrak{P}\in S(p)\mid  d(\mathfrak{P})=m\}\subset S(p).
\end{equation}
Combining \eqref{PsiD} with \eqref{dSUMrho}, we obtain that
\begin{equation}
\label{PsiDrho}
\rho(\Psi(\pi)):=\{\mathfrak{P}\in S(p)\mid  d(\mathfrak{P})=0\}=\{\mathfrak{P}\in S(p)\mid   \frac{\mathrm{ord}_{\mathfrak{P}}(\pi)}{\mathrm{ord}_{\mathfrak{P}}(q)}=0\}\subset S(p);
\end{equation}
in particular, the subsets $\Psi(\pi)$ and $\rho(\Psi(\pi))$ do {\sl not meet} each other.
In light of \eqref{PsiPiOrd} and \eqref{PsiDrho} combined with \eqref{ordinary01},
$\pi$ is ordinary if and only if $S(p)$ is a disjoint
union of $\Psi(\pi)$ and $\rho(\Psi(\pi))$, i.e., $\Psi(\pi)$ is a $p$-type. This proves the equivalence of (i) and (iii).
If (i) and (iii) hold then it follows from \eqref{piProd} that
$$\pi \OO_K=\prod_{\mathfrak{P}\in \Psi(\pi)}\mathfrak{P}^m=\mathfrak{B}^m \ \text{ where } \mathfrak{B}:=\prod_{\mathfrak{P}\in \Psi(\pi)}\mathfrak{P}.$$
Since $\Psi(\pi)$ is a $p$-type, $\mathfrak{B}\in H(p)$ and obviously $\Phi(\mathfrak{B})=\Psi(\pi)$. This proves that equivalent (i) and (iii) imply (ii).

Let us assume that (ii) holds. This means that there is $\mathfrak{B} \in H(p)$ that satisfies \eqref{HpPi}.  This implies that
$$\mathfrak{B} =\prod_{\mathfrak{P}\in \Phi(\mathfrak{B})}\mathfrak{P}, \ \pi\OO_K=\mathfrak{B}^m=\prod_{\mathfrak{P}\in \Phi(\mathfrak{B})}\mathfrak{P}^m.$$
It follows that 
$$\frac{\mathrm{ord}_{\mathfrak{P}}(\pi)}{\mathrm{ord}_{\mathfrak{P}}(q)}=1 \ \forall \mathfrak{P}\in \Phi(\mathfrak{B}),$$
$$\frac{\mathrm{ord}_{\mathfrak{P}}(\pi)}{\mathrm{ord}_{\mathfrak{P}}(q)}=0 \ \forall \mathfrak{P}\not\in \Phi(\mathfrak{B}).$$
This implies that $\pi$ is ordinary and therefore (ii) implies (i). So, we have proven the equivalence of (i),(ii), (iii).
The uniqueness of such $\mathfrak{B} $ is obvious.
\end{proof}

\begin{lem}
\label{transH}
The natural action of $G$ on $H(p)$ is free. In particular, $H(p)$ partitions into a disjoint union
of $2^{2^{n-1}-n}$ orbits of $G$, each of which consists of $2^n$ elements.
\end{lem}

\begin{proof}
Suppose that there exists $\mathfrak{B}\in H(p)$ such that its stabilizer 
$$G_{\mathfrak{B}}=\{\sigma \in G\mid \sigma(\mathfrak{B})=\mathfrak{B}\}$$
is a nontrivial subgroup. Then $G_{\mathfrak{B}}$ must contain $\rho$, thanks to the arguments of Subsection \ref{Gideals}. This means that
$\rho(\mathfrak{B})=\mathfrak{B}$ and therefore
$$p\cdot \OO_K=\mathfrak{B}\cdot \rho(\mathfrak{B})=\mathfrak{B}^2,$$
which is not true, since $p$ is unramified in $K$. The obtained contradiction proves that the action of $G$ on $H(p)$ is free.
Hence, every $G$-orbit in $H(p)$ consists of $\#(G)=2^n$ elements and the number of such orbits is
$$\frac{\#(H(p))}{\#(G)}=\frac{2^{2^{n-1}}}{2^n}=2^{2^{n-1}-n}.$$
\end{proof}

In what follows we define (non-canonically) certain $G$-equivariant injective maps $\mathcal{Z}$,  $\Pi$ and $\Pi_1$ from $H(p)$ to $K$;  they will play a crucial role in the classification of ordinary Weil's numbers in $K$ up to equivalence.

\begin{cor}
\label{gen}
Let $c=\exp(K)$. Then there exists a $G$-equivariant map
\begin{equation}
\label{generator}
\mathcal{Z}: H(p) \hookrightarrow \OO_K \setminus\{0\}\subset \OO_K\subset K
\end{equation}
such that $\mathcal{Z}(\mathfrak{B})$ is a generator of  $\mathfrak{B}$
for all $\mathfrak{B}\in H(p)$.
\end{cor}

\begin{proof}
We define $\mathcal{Z}$ separately for each $G$-orbit $O\subset H(p)$.
Pick $\mathfrak{B}_O \in O$ and choose a generator $z_O$ of the principal ideal $\mathfrak{B}_O^c$.
In light of Lemma \ref{transH}, if $\mathfrak{B} \in O$ then there is precisely one $\sigma \in G$ such that
$\mathfrak{B}=\sigma(\mathfrak{B}_O)$. This implies that
$$\mathfrak{B}^c=\sigma(\mathfrak{B}_O)^c=\sigma(\mathfrak{B}_O^c)=\sigma(z_O)\OO_K,$$
i.e., $g(z_O)$ is a generator of $\mathfrak{B}^c$. It remains to put
$$\mathcal{Z}(\mathfrak{B}):=\sigma(z_O).$$
\end{proof}

\begin{thm}
\label{HpOrdinary}
Let us put
$$c:=\exp(K), \ q:=p^c.$$
Let $K_0=K^{\rho}$ be the maximal totally real subfield of $K$.
\begin{enumerate}
\item[(1)]
There exists an injective 
map
\begin{equation}
\label{theMaP}
\Pi: H(p) \hookrightarrow W(q^2,K),  \ \mathfrak{B} \mapsto \Pi(\mathfrak{B})
\end{equation}
that enjoys the following properties.

\begin{itemize}
\item[(0)] $\Pi$ is $G$-equivariant, i.e.,
$$\Pi(\sigma(\mathfrak{B}))=\sigma(\Pi(\mathfrak{B})) \ \forall \sigma\in G, \mathfrak{B}\in H(p).$$
\item[(i)] For all $\mathfrak{B} \in H(p)$ the ideal $\Pi(\mathfrak{B})\OO_K$ coincides with $\mathfrak{B}^{2 c}$.
\item[(ii)]
The image $\Pi(H(p))$ consists of ordinary Weil's $q^2$-numbers.
\item[(iii)]
If $\pi^{\prime}$ is an ordinary Weil's $p^m$-number in $K$ then there exists prcisely one $ \mathfrak{B} \in H(p)$
such that  the ratio $(\pi^{\prime})^{2c}/\Pi(\mathfrak{B})^m$ is a root of unity.
\item[(iv)]
Let  $\mathfrak{B}_1,  \mathfrak{B}_2 \in H(p)$. Then Weil's $q^2$-numbers $\Pi(\mathfrak{B}_1)$ and $\Pi(\mathfrak{B}_2)$
are equivalent if and only if $\mathfrak{B}_1$ and  $\mathfrak{B}_2$ lie in the same $G$-orbit.
\item[(v)]
If $h$ is a positive integer then the subfield $\Q\left(\Pi(\mathfrak{B})^h\right)$ of $K$ generated by $\Pi(\mathfrak{B})^h$ coincides with $K$.
\end{itemize}
\item[(vi)]
Suppose that every totally positive unit in $\UU_{K_0}$ is a square in $K_0$ (anf therefore
in $\UU_{K_0}$).
Then there exists a map
$$\Pi_0: H(p) \to W(q,K)$$ that enjoys the following properties.

\begin{itemize}
\item[(a)]
$\Pi_0(\mathfrak{B})^2=\Pi(\mathfrak{B})$ for all $\mathfrak{B}$.
\item[(b)]
$\Pi_0$ is $G$-equivariant ``up to sign'', i.e.,
$$\Pi_0(\sigma(\mathfrak{B}))=\pm \sigma(\Pi_0(\mathfrak{B})) \ \forall \sigma\in G, \mathfrak{B}\in H(p).$$
\item[(c)]
If $h$ is a positive integer then the subfield $\Q\left(\Pi_0(\mathfrak{B})^h\right)$ of $K$ generated by $\Pi(\mathfrak{B})^h$ coincides with $K$
\item[(d)] $\Pi_0(\mathfrak{B})$ is an ordinary Weil's $q$-number for all $\mathfrak{B}\in H(p)$.
\end{itemize}
\end{enumerate}
\end{thm}

\begin{proof}
Let us choose
 $\mathcal{Z}: H(p) \to \OO_E \setminus\{0\}$ that enjoys the properties described in Corollary \ref{gen}.
Let $\mathfrak{B}\in H(p)$. In order to define $\Pi(\mathfrak{B})$, notice that
$$\mathfrak{B}\cdot \rho(\mathfrak{B})=p\OO_K, \  \mathfrak{B}^c=z \OO_K$$
where 
\begin{equation}
\label{zZB}
z=\mathcal{Z}(\mathfrak{B}) \in \OO_K\setminus \{0\}.
\end{equation} 
Then
$z\rho(z)$ is a generator of the ideal
$$\mathfrak{B}^c\cdot \rho(\mathfrak{B}^c)=\left(\mathfrak{B}\cdot \rho(\mathfrak{B})\right)^c=p^c\cdot\OO_K=q\OO_K.$$
Since $\rho$ is the complex conjugation, $z\rho(z)$ is a real (i.e., $\rho$-invariant) totally positive element of $\OO_K$.  Clearly,
$$u:=\frac{z\rho(z)}{q}$$
is an invertible element of $\OO_K$ that is also $\rho$-invariant and totally positive unit in $\UU_{K_0}$. Obviously,
$$q=\frac{z\cdot \rho(z)}{u}.$$
Now let us put
\begin{equation}
\label{WeilQ}
\Pi(\mathfrak{B}):=q \cdot \frac{z}{\rho(z)}=\frac{z^2}{z\rho(z)/q}=\frac{z^2}{u} \in \OO_K.
\end{equation}
If $u$ is a square in $K_0$ then there is a unit $u_0$ in $K_0$ such that $u=u_0^2$. If this is the case then let us put
\begin{equation}
\label{WeilQ1}
\Pi_0(\mathfrak{B}):=\frac{z}{u_0}\in \OO_K \ \text{ and get } \Pi_0(\mathfrak{B})^2=\left(\frac{z}{u_0}\right)^2=\frac{z^2}{u}=\Pi(\mathfrak{B}).
\end{equation}
Clearly,
\begin{equation}
\label{idealPi}
\Pi(\mathfrak{B})\cdot \OO_K=z^2\cdot \OO_K=(z\cdot\OO_K)^2=(\mathfrak{B}^{c})^2=\mathfrak{B}^{2c},
\end{equation}
which proves (i).  In order to check that $\Pi(\mathfrak{B})$ is a Weil's $q^2$-number, notice that
$$\Pi(\mathfrak{B})\cdot  \rho(\Pi(\mathfrak{B}))=q \cdot \frac{z}{\rho(z)}\cdot \rho\left(q \cdot \frac{z}{\rho(z)}\right)=q^2\cdot  \frac{z}{\rho(z)}
\cdot \frac{\rho(z)}{z}=q^2.$$
In light of Remark \ref{useful}, this proves that $\Pi(\mathfrak{B})$ is a Weil's $q^2$-number. It follows from \eqref{WeilQ1} that if $\Pi_0(\mathfrak{B})$ is defined then it is a Weil's $q$-number.
 By construction,
$$\Pi(\mathfrak{B})\OO_K=\mathfrak{B}^{2c},$$
which also implies that $\Pi(\mathfrak{B})$ is $p^{2c}=q^2$-ordinary Weil's number. The $G$-invariance of $\mathcal{Z}$ (see Corollary \ref{gen})
combined with \eqref{zZB} and \eqref{WeilQ} implies the $G$-equivariance of $\Pi$, which proves (0). The injectiveness of $\Pi$ follows from \eqref{idealPi}. This proves (i) and (ii).

In order to prove (v), notice that if $\Q\left(\Pi(\mathfrak{B})^h\right)$ does {\sl not} coincide with $K$ then it consists of $\rho$-invariants (Subsection \ref{Gideals}). In particular, the ideal
$\Pi(\mathfrak{B})^h\OO_K=\mathfrak{B}^{2ch}$ coincides with its complex-conjugate
$$\rho\left(\Pi(\mathfrak{B})^h\OO_K\right)=\rho\left(\mathfrak{B}^{2ch}\right)=\rho(\mathfrak{B})^{2ch}.$$
This implies that  $\mathfrak{B}=\rho(\mathfrak{B})$, which is not the case, since $\mathfrak{B}\in H(p)$. The obtained contradiction proves (v).

In order to prove (iii), we need to check that if $\pi^{\prime}$ is an ordinary Weil's $p^m$-number in $K$ then it is equivalent to $\Pi(\mathfrak{B})$ for some $\mathfrak{B}\in H(p)$.
In order to do that, let us consider the ideal
$\mathfrak{M}:=\pi^{\prime}\OO_K$ in $\OO_K$. Since $\pi^{\prime}\cdot\rho(\pi^{\prime})=p^m$, we get
$\mathfrak{M}\cdot \rho(\mathfrak{M})=p^m \OO_K$. It follows that
$$\mathfrak{M}=\prod_{\mathfrak{P}\in S(p)}\mathfrak{P}^{d(\mathfrak{P})}, \ d(\mathfrak{P})+d(\mathfrak{\rho(P)})=m \ \forall \mathfrak{P}\in S(p).$$
The ordinarity of $\mathfrak{M}$ implies that 
$$d(\mathfrak{P})=0 \ \text{ or } m \ \forall \mathfrak{P}\in S(p).$$
This implies that if we put
$$\Phi=\{\mathfrak{P}\in S(p)\mid d(\mathfrak{P})=m\} \subset S(p)$$
then $\Phi$ is a $p$-type and
$$\mathfrak{M}=\prod_{\mathfrak{P}\in \Phi}\mathfrak{P}^m=\left(\prod_{\mathfrak{P}\in \Phi}\mathfrak{P}\right)^m.$$
It is also clear that 
$$\mathfrak{B}:=\prod_{\mathfrak{P}\in \Phi}\mathfrak{P}\in H(p), \  $$
and
$$(\pi^{\prime})^{2c}\OO_K=\mathfrak{M}^{2c}=\mathfrak{B}^{2cm}=\left(\mathfrak{B}^{2c}\right)^m=\left(\Pi((\mathfrak{B})\OO_K\right)^m=\Pi\left(\mathfrak{B}^m\right)\OO_K.$$
It follows from Remark \ref{kron} that the ratio  $\Pi(\mathfrak{B})^m /(\pi^{\prime})^{2c}$ is a root of unity. The uniqueness follows from the already
proven (i).

Let us prove (iv). The already proven (0)  tells us that  if $\mathfrak{B}_2 =\sigma(\mathfrak{B}_2)$
for $\sigma \in G$ then $\Pi(\mathfrak{B}_2)=\sigma(\Pi(\mathfrak{B}_1)$ and therefore Weil's numbers 
$\Pi(\mathfrak{B}_1)$ and $\Pi(\mathfrak{B}_2)$ are equivalent.  

Conversely, suppose that $\Pi(\mathfrak{B}_1)$ and $\Pi(\mathfrak{B}_2)$ are equivalent. This means that there are positive integers $a,b$, a Galois automorphism
$\sigma \in G$, and a root of unity $\zeta \in \mu_K$ such that
$$\Pi(\mathfrak{B}_2)^a=\zeta \cdot  \sigma(\Pi(\mathfrak{B}_1))^b.$$
This implies the equality of the corresponding ideals in $\OO_K$:
$$\Pi(\mathfrak{B}_2)^a\OO_K=\sigma(\Pi(\mathfrak{B}_1))^b\OO_K=\Pi(\sigma(\mathfrak{B}_1))^b.$$
This means (in light of already proven (i)) that
$$\mathfrak{B}_2^{2 ca}=(\sigma(\mathfrak{B}_1))^{2cb},$$
which implies $\mathfrak{B}_2=\sigma(\mathfrak{B}_1)$. Hence  $\mathfrak{B}_1$ and  $\mathfrak{B}_2$ lie in the same $G$-orbit.

Let us prove (vi). Actually, we have already constructed the map
$\Pi_0: H(p) \to \OO_K$, checked that its image lies in $W(q,K)$; we have also proven property (vi)(a). 
As for (vi)(b), it follows readily from \eqref{WeilQ1} combined with the $G$-equivariance of $\Pi$.  As for (vi)((c), it follows readily from (v) combined with \eqref{WeilQ1}. 
In order to prove (vi)(d), it suffices to recall that $\Pi(\mathfrak{B})$ is an ordinary Weil's $q^2$-number and notice that in light of \eqref{WeilQ1}, the integer
$$\frac{\ord_{\mathfrak{P}}(\Pi(\mathfrak{B})}{\ord_{\mathfrak{P}}(q^2)}=
\frac{2\ord_{\mathfrak{P}}(\Pi_0(\mathfrak{B})}{2\ord_{\mathfrak{P}}(q)}=\frac{\ord_{\mathfrak{P}}(\Pi_0(\mathfrak{B})}{\ord_{\mathfrak{P}}(q)}.$$

\end{proof}

\section{Abelian varieties with Weil's numbers in $K$}
\label{WeilKK}
As above, $p$ is a prime, $m$ a positive integer and $q=p^m$.

\begin{thm}
\label{ordinaryAV}
Let $A$ be a simple abelian variety over $k=\F_q$ such that the corresponding Weil's $q$-number
$$\pi_A \in K.$$
Let $\Q(\pi_A)$ be the subfield of $K$ generated by $\pi_A$.

\begin{itemize}
\item[(i)]
Suppose that either $\Q(\pi_A)\ne K$ or $p$ does not split completely in $K$.

 Then $A$ is supersingular.
\item[(ii)]
If $p$ splits completely in $K$, $\Q(\pi_A)=K$ and $\pi_A$ is not ordinary then
the division $\Q$-algebra
$\End_k^{0}(A)$
is not commutative.
\item[(iii)]
If $\pi_A$ is ordinary then $K=\Q(\pi_A)$,
and $\End_k^{0}(A) \cong K$; in particular,  $\End_k^{0}(A)$  is commutative.
\end{itemize}
\end{thm}

\begin{proof}

(i) It follows from   Lemmas \ref{nonsplit} and \ref{PsplitK} that  $\pi_A^2/q$ is a root of unity. This means that $A$ is supersingular.

(ii-iii) Recall \cite{T66,T} that  $E:=\End_k^{0}(A)$  is a {\sl central} division algebra  over the field $\Q(\pi_A)=K$.
 Since $p$ splits completely in $K$,
the $\mathfrak{P}$-adic completion $K_{\mathfrak{P}}$ of $K$ coincides with $\Q_p$, i.e.,
$$[K_{\mathfrak{P}}:\Q_p]=1 \ \forall \mathfrak{P} \in S(p).$$
By \cite[Th. 1]{T}, the local $\mathfrak{P}$-adic invariant 
$$\mathrm{inv}_{\mathfrak{P}}(E)\in \Q/\Z$$
 of the central division $K$-algebra $E$
  is given by the formula 
\begin{equation}
\label{localInv}
\mathrm{inv}_{\mathfrak{P}}(E)= \frac{\mathrm{ord}_{\mathfrak{P}}(\pi_A)}{\mathrm{ord}_{\mathfrak{P}}(q)}[K_{\mathfrak{P}}:\Q_p]\bmod \Z= 
 \frac{\mathrm{ord}_{\mathfrak{P}}(\pi_A)}{\mathrm{ord}_{\mathfrak{P}}(q)}\bmod \Z\in \Q/\Z.
 \end{equation}
 All other local invariants of $E$ (outside $S(p)$)
  are $0$ (ibid).
 
Suppose that  $\pi_A$ is  {\sl ordinary}.  Then $\Q(\pi_A)=K$, because otherwise $\Q(\pi_A)\subset \R$ and therefore $\pi_A$ is  {\sl real}, i.e., $A$ is {\sl supersingular} \cite[Examples]{T},
which is not the case. Since  $\pi_A$ is  ordinary,
all  the {\sl slopes} $\mathrm{ord}_{\mathfrak{P}}(\pi_A)/\mathrm{ord}_{\mathfrak{P}}(q)$ are {\sl integers}
and therefore $\mathrm{inv}_{\mathfrak{P}}(E)=0$ for all $\mathfrak{P}\in S(p)$. This implies that the division algebra $E=\End_k^{0}(A)$  is actually a {\sl field},
i.e., is isomorphic to $K$. This proves (iii).

In order to prove (ii), assume  that   $\pi_A$ is  {\sl not ordinary}. Then there is a maximal ideal 
$\mathfrak{P} \in S(p)$ such that the ratio
$\mathrm{ord}_{\mathfrak{P}}(\pi_A)\mathrm{ord}_{\mathfrak{P}}(q)$ is {\sl not an integer}, i.e.
\begin{equation}
\label{modZ}
\frac{\mathrm{ord}_{\mathfrak{P}}(\pi_A)}{\mathrm{ord}_{\mathfrak{P}}(q)}\bmod \Z \ne 0 \ \text{ in  }
\Q/\Z.
\end{equation}
Combining  \eqref{modZ} with \eqref{localInv}, we obtain that
$\mathrm{inv}_{\mathfrak{P}}(E) \ne 0.$
 It follows that $E=\End_k^{0}(A)$ does {\sl not} coincide with its center, i.e., is {\sl noncommutative}. This proves (ii).
\end{proof}

\begin{rem}
\label{ordinarySimple}
Let $A$ be a simple abelian variety over $\F_q$ such that $\pi_A \in K$.
Obviously, $A$ is ordinary if and only if $\pi_A$ is ordinary.
\end{rem}

\section{Honda-Tate theory for ordinary abelian varieties}
\label{HondaTate}

As above, $p$ is a prime that splits completely in $K$, $m$ a positive integer and $q=p^m$.

Let $\pi \in K$ be a Weil's $q$-number. The Honda-Tate theory \cite{T66,Honda,T} attaches to $\pi$ a {\sl simple} abelian variety  $\mathcal{A}$ over $\F_{q}$ that is defined up to an $\F_{q}$-isogeny 
and enjoys the following properties.

Let $\mathrm{Fr}_{\mathcal{A}}: \mathcal{A}\to \mathcal{A}$ be the Frobenius endomorphism of $\mathcal{A}$ and
$F:=\Q[\mathrm{Fr}_{\mathcal{A}}]$ be the $\Q$-subalgebra of the division $\Q$-algebra  $E:=\End_{\F_{q}}^{0}(\mathcal{A})$
(which is actually a subfield). Then $F$ is the {\sl center} of $E$ and
there is a field embedding
$$i: F \hookrightarrow \C \ \text{ such that } \ i(\mathrm{Fr}_{\mathcal{A}})=\pi.$$

\begin{lem}
\label{HTlemma}
 Suppose $\pi$ is ordinary  and $\Q(\pi^h)=K$ for all positive integers $h$. Then $\mathcal{A}$ is an absolutely simple $2^{n-1}$-dimensional  ordinary abelian variety,  $\End^0(A) \cong K$,  and all  endomorphisms of $\mathcal{A}$ are defined over $\F_{q}$.
\end{lem}

\begin{proof}
Since $\Q(\pi)=K$, we get $i(F)=K$. In particular, number fields $K$ and $F$ are isomorphic. In light of Theorem \ref{ordinaryAV},  $\mathcal{A}$ is an ordinary abelian variety
with commutative endomorphism algebra
$E=F\cong K$. By Theorem 2(c) of \cite[Sect. 3]{T66},
$$\dim(\mathcal{A})=\frac{[E:\Q]}{2}=\frac{[K:\Q]}{2}=2^{n-1}.$$
We are going to prove that $\mathcal{A}$ is absolutely simple and all its endomorphisms are defined over $\F_{q}$. Let $h$ be a positive integer and
$k =\F_{q^{h}}$  a degree $h$ field extension of $\F_{q}$. Let $\mathcal{A}_k=\mathcal{A}\times_{\F_{q}}k$ be the abelian variety over $k$ obtained from $\mathcal{A}$
by the extension of scalars.  There is the natural embedding (inclusion) of $\Q$-algebras
$$\End_{\F_{q}}^{0}(\mathcal{A})\subset \End_k^0(\mathcal{A}_k)$$
such that the Frobenius endomorphism $\Fr_{\mathcal{A}_k}$ coincides with $\Fr_{\mathcal{A}}^h$. In particular,
$$\Q[\Fr_{\mathcal{A}_k}]\subset \Q[\mathrm{Fr}_{\mathcal{A}}]=F.$$
In addition, 
$$i(\Fr_{\mathcal{A}_k})=i\left(\Fr_{\mathcal{A}}^h\right)= i\left(\Fr_{\mathcal{A}}\right)^h=\pi^h.$$
Since $\Q[\pi^h]=K=\Q(\pi)$, we get
$$i(\Q[\Fr_{\mathcal{A}_k}])=K=i(\Q[\mathrm{Fr}_{\mathcal{A}}]).$$ 
Hence,
$\Q[\Fr_{\mathcal{A}_k}]=\Q[\mathrm{Fr}_{\mathcal{A}}]$ is a number field of degree $2\dim(\mathcal{A})=2\dim(\mathcal{A}_k)$.
Applying again Theorem 2(c) of \cite[Sect. 3]{T66} to $\mathcal{A}_k$, we conclude that
$$\End^0(\mathcal{A}_k)=\Q[\Fr_{\mathcal{A}_k}]=\Q[\mathrm{Fr}_{\mathcal{A}}]=\End_{\F_{q}}^{0}(\mathcal{A})$$
for all finite overfields $k$ of $\F_{q}$. This implies that 
$$\End^0(\mathcal{A}_k)=\End_{\F_{q}}^{0}(\mathcal{A}),$$
 i.e.,
all the endomorphisms of $\mathcal{A}$ are defined over $\F_{q}$. In particular, $\mathcal{A}$  is absolutely simple and $\End^0(\mathcal{A}) \cong K$.

\end{proof}

\section{Proofs of main results}
\label{mainProof}
As above, $c=\exp(K)$, a prime $p$ splits completely in $K$ and $q=p^c$.

\begin{proof}[Proof of Theorem \ref{main}]
Let $\Pi: H(p) \to W(q^2,K)$ be as in Theorem \ref{HpOrdinary}.
Let $\mathcal{B}\in H(p)$ and let $\Pi(\mathcal{B})$ be the corresponding ordinary Weil's $q^2$-number in $K$.
In light of  Theorem \ref{HpOrdinary}(v),  $\Q[\Pi(\mathfrak{B})^h]=K$ for all positive integers $h$.
In light of Lemma \ref{HTlemma} applied to $q^2$ and $\Pi(\mathfrak{B})$, the Honda-Tate theory \cite{T66,Honda,T} attaches to $\Pi(\mathcal{B})$ an {\sl absolutely simple} $2^{n-1}$-dimensional abelian variety 
 $\mathcal{A}=A(\mathcal{B})$ over $\F_{q^2}$ (that is defined up to an $\F_{q^2}$-isogeny) 
such that $\End^0(A(\mathcal{B})\cong K$, and all endomorphisms  of $A(\mathcal{B})$ are defined over $\F_{q^2}$.

By Theorem \ref{HpOrdinary}(iv),  if $\mathfrak{B}_1,  \mathfrak{B}_2 \in H(p)$ then the
Weil numbers $\Pi(\mathfrak{B}_1)$ and $\Pi(\mathfrak{B}_2)$ are {\sl equivalent} if and only if
 $ \mathfrak{B}_1,$ and $\mathfrak{B}_2$ belong to the same $G$-orbit. In light of \cite[Theorem 1]{T66}, \cite[p. 84] {Honda}
combined with Lemma \ref{transH}, 
all the $A(\mathfrak{B})$ lie in precisely $2^{2^{n-1}-n}$  isogeny classes of abelian varieties over $\bar{\F}_p$. We also know that each of these varieties
is ordinary, has dimension $2^{n-1}$ and their endomorphism algebras are isomorphic to $K$. 

Now, let us prove that each abelian variety $\mathcal{B}$ over 
$\bar{\F}_p$, whose endomorphism algebra is isomorphic to $K$, is isogenous to one of  $A(\mathfrak{B})$ over $\bar{\F}_p$,

In order to do that, first, notice that since $K$ is a field,    $\mathcal{B}$  is simple over $\bar{\F}_p$.   Second,  $\mathcal{B}$
is defined with all its endomorphisms over a  certain finite field $k=\F_{q^{2h}}$
(where $h$ is a certain positive integer), i.e.,
there is a simple abelian variety $\mathcal{B}_k$ over $k$ such that 
$$\mathcal{B}=\mathcal{B}_k\times_k \bar{\F}_p,  \ \End_k^{0}(\mathcal{B}_k)=\End^0(\mathcal{B}) \cong K.$$
Applying  Theorem 2(c) of \cite[Sect. 3]{T66} to $\mathcal{B}_k$, we get
$$K \cong \End^0(\mathcal{B})= \End_k^0(\mathcal{B}_k)=\Q[\Fr_{\mathcal{B}_k}]$$
where $\Fr_{\mathcal{B}_k}$ is the Frobenius endomorphism of $\mathcal{B}_k$. This gives us a field isomorphism
$\Q[\Fr_{\mathcal{B}_k}] \to K$; let us denote by $\pi_{\mathcal{B}_k}$ the image of $\Fr_{\mathcal{B}_k}$ in $K$.
Clearly, $\Q(\pi_{\mathcal{B}_k})=K$; according to a classical result of Weil \cite{Mumford}, $\pi_{\mathcal{B}_k}$ is a Weil's 
$q^{2h}$-number. By Theorem \ref{ordinaryAV}(i) (applied to $q^{2h}$ instead of $q$),  $\pi_{\mathcal{B}_k}$ is {\sl ordinary},
because $\End_k^0(\mathcal{B}_k) \cong K$ is {\sl commutative}. It follows from  Theorem \ref{HpOrdinary}(iii) that there is $\mathfrak{B} \in H(p)$
such that Weil's numbers $\pi_{\mathcal{B}_k}$  and $\Pi(\mathfrak{B})$ are {\sl equivalent}. This means (thanks to  Theorem 1 of \cite{T66}, 
see also \cite[pp. 83--84]{Honda}) that {\sl absolutely simple} abelian varieties
$\mathcal{B}_k$ and  $A(\mathfrak{B})$ become isogenous over $\bar{\F}_p$. 
It follows that  {\sl absolutely simple} abelian varieties $\mathcal{B}=\mathcal{B}_k\times_k \bar{\F}_p$
and $A(\mathfrak{B})$ are isogenous over $\bar{\F}_p$. 

This proves (i), (ii)(1) and (ii)(2).
It remains to prove (ii)(3).  It suffices to check that for each $\mathcal{B}\in H(p)$ 
there exists an abelian variety $A_0$ that is defined over $\F_q$ with all its endomorphism and such that
 $A(\mathcal{B})$ is isogenous to $A_0$ over $\bar{\F}_p$.

Let $\Pi_0: H(p) \to W(q,K)$ be as in Theorem \ref{HpOrdinary}(vi)
 and $\Pi_0(\mathcal{B})$ be the corresponding ordinary Weil's $q$-number in $K$.
In light of  Theorem \ref{HpOrdinary}(vi)(c),  $\Q[\Pi_0(\mathfrak{B})^h]=K$ for all positive integers $h$.
In light of Lemma \ref{HTlemma} applied to $q$ and $\Pi_0(\mathfrak{B})$, the Honda-Tate theory \cite{T66,Honda,T} attaches to  Weil's $q$-number $\Pi_0(\mathcal{B})$ an {\sl absolutely simple} $2^{n-1}$-dimensional abelian variety 
 $\mathcal{A}_0$ over $\F_{q}$ (that is defined up to an $\F_{q}$-isogeny) 
such that $\End^0(\mathcal{A}_0)\cong K$, and all endomorphisms  of $\mathcal{A}_0$ are defined over $\F_{q}$.

Since  $\Pi_0(\mathcal{B})^2=\Pi(\mathcal{B})$, Weil's numbers $\Pi_0(\mathcal{B})$ and $\Pi(\mathcal{B})$ are {\sl equivalent}.
As above, in light of Theorem 1 of \cite{T66} 
(see also \cite[pp. 83--84]{Honda}), the corresponding {\sl absolutely simple} abelian varieties $\mathcal{A}_0$ and $A(\mathcal{B})$ are isogenous over $\bar{\F}_p$.  
This ends the proof.

\end{proof}

\begin{proof}[Proof of Corollary \ref{cycleP}]
Recall that $r$ is an odd prime and $\zeta_r$ is a primitive $r$th root of unity.
Clearly, $\Q(\zeta_r)$ is a CM field. Hence, its subfield $K$ is either CM or a totally real.
Since $\HH$ has odd order $m$, it does {\sl not} contain the complex conjugation
$\rho: \Q(\zeta_r) \to \Q(\zeta_r),$ because $\rho$ has order $2$. Hence, $\rho$ acts nontrivially on $K=\Q(\zeta_r)^{\HH}=K^{(r)}$,
which implies that $K$ is a CM field. (See also \cite[p. 78]{CH}.) Its degree
$$[K:\Q]=\frac{[\Q(\zeta_r):\Q]}{\#(\HH)}=\frac{m \cdot 2^n}{m}=2^n.$$
We also know (Remark \ref{Kr0}) that every totally positive unit in $K_0$ is a square in $K_0$.

Clearly, $K/\Q$ is ramified at $r$ and unramified at every prime $p \ne r$. Let us find which $p\ne r$ split completely in $K$.
Let  
$$f_p \in \Gal(\Q(\zeta_r)/\Q)=(\Z/r\Z)^{*}$$
be the Frobenius element attached to $p$, which is characterized by the property
$$f_p(\zeta_r)=\zeta_r^p.$$
In other words,
$$f_p=p \bmod \ r\in (\Z/r\Z)^{*}=\Gal(\Q(\zeta_r)/\Q).$$
Clearly, $p$ splits completely in $K$ if and only if $f_p\in \HH$. So,
we need to find when  $f_p$ lies in $\HH$. In order to do it, notice that 
$$\HH=\{\sigma^{2^n}\mid \sigma \in \Gal(\Q(\zeta_r)/\Q)= (\Z/r\Z)^{*}\}.$$
This implies that $f_p$ lies in $\HH$ if and only if $p \bmod \ r$ is a $2^r$th power in $\Z/r\Z=\F_r$.
This ends the proof of  (0).

The remaining assertions (i) and (ii) follow from Theorem \ref{main} combined with (0).

\end{proof}

\begin{proof}[Proof of Corollary \ref{mainFermat}]
In the notation of Corollary \ref{cycleP}, this is the case when $m=1$ and $2^n=r-1$.
By little Fermat's theorem, every nonzero $a \in \Z/r\Z$ satisfies 
$$a^{2^n}=a^{r-1}=1.$$
Now the desired result follows readily from Corollary \ref{cycleP}.

\end{proof}

\end{document}